\documentclass[12pt]{article}
\usepackage[margin=2cm]{geometry}

\usepackage{amsmath}
\usepackage{amsfonts}
\usepackage{amssymb}
\usepackage{amsthm}
\usepackage{thmtools}
\usepackage{cases}
\usepackage{enumerate}
\newtheorem{thm}{Theorem}

\newtheorem{lem}[thm]{Lemma}

\newtheorem{conj}[thm]{Conjecture}
\theoremstyle{definition}

\numberwithin{equation}{section}

\makeatletter
\g@addto@macro\bfseries{\boldmath}
\makeatother

\def\eref#1{$(\ref{#1})$}
\def\sref#1{\S$\ref{#1}$}
\def\lref#1{Lemma~$\ref{#1}$}
\def\tref#1{Theorem~$\ref{#1}$}

\newcommand{\Tr}{\mathop{\mathrm{Tr}}}

\renewcommand{\geq}{\geqslant}
\renewcommand{\leq}{\leqslant}
\renewcommand{\ge}{\geqslant}
\renewcommand{\le}{\leqslant}

\def\F{\mathbb{F}}

\title{Degree of Orthomorphism Polynomials\\ over Finite Fields}
\author{Jack Allsop \ \ Ian M. Wanless\thanks{Research supported by ARC grant DP150100506.}\\
\small School of Mathematics\\[-0.5ex]
\small Monash University\\[-0.5ex]
\small Vic 3800, Australia\\
\small\tt jall0007@student.monash.edu \ \ ian.wanless@monash.edu}

\date{}
\begin{document}

\maketitle

\begin{abstract}
An \emph{orthomorphism} over a finite field $\F_q$ is a permutation
$\theta:\F_q\mapsto\F_q$ such that the map $x\mapsto\theta(x)-x$ is
also a permutation of $\F_q$.  The \emph{degree} of an orthomorphism
of $\F_q$, that is, the degree of the associated reduced permutation
polynomial, is known to be at most $q-3$. We show that this upper
bound is achieved for all prime powers $q\notin\{2, 3, 5, 8\}$. We do
this by finding two orthomorphisms in each field that differ on
only three elements of their domain. Such orthomorphisms can be used
to construct $3$-homogeneous Latin bitrades.

\medskip

\noindent Keywords: orthomorphism, cyclotomic, permutation polynomial,
homogeneous Latin bitrades.
\end{abstract}

\section{Introduction}

It is well known that any map $\phi:\F_q\rightarrow \F_q$ can be
expressed uniquely as a polynomial $f \in \F_q[x]$ of degree less than
$q$. We say that $f$ is the \emph{reduced polynomial} corresponding to $\phi$
and that the \emph{reduced degree} of $\phi$ is the degree of $f$.
A polynomial $f \in \F_q[x]$ is a \emph{permutation polynomial} if the map $x \mapsto f(x)$ is a permutation of $\F_q$. For $q>2$ it is well known that the reduced degree of a permutation polynomial is at most $q - 2$, and using Lagrange interpolation it is easily verified that any transposition has reduced degree exactly $q - 2$.

A permutation polynomial $f \in \F_q[x]$ is an \emph{orthomorphism polynomial} if the map $x \mapsto f(x) - x$ is also a permutation of $\F_q$. Orthomorphisms have many applications in design theory, especially to Latin squares \cite{evans2018orthogonal,wanless2004diagonally}. The following theorem was proven by Niederreiter and Robinson \cite{niederreiter1982complete} for fields of odd characteristic, and by Wan \cite{wan1986problem} for fields of even characteristic. 

\begin{thm}\label{t:reddegorth}
For $q>3$ any orthomorphism polynomial over $\F_q$ has reduced
degree at most $q - 3$.
\end{thm}

Our first goal is to establish when the bound in \tref{t:reddegorth}
is achieved. 
It was known \cite{shallue2013permutation} that the bound in \tref{t:reddegorth} is not achieved when $q\in\{2, 3, 5, 8\}$.
We show:

\begin{thm}\label{t:bndach}
There exists an orthomorphism polynomial of degree $q - 3$ over $\F_q$ if and only if $q \notin \{2, 3, 5, 8\}$.
\end{thm}

We define the \emph{Hamming distance} $H(f,g)$ between two polynomials
$f,g\in\F_q[x]$ by $H(f,g)=\big|\{a\in\F_q:f(a)\ne g(a)\}\big|$.
For two distinct permutations $f,g$ it is obvious that $H(f,g)\ge2$.
If $f(a)\ne f(b)$ for $a\ne b$, then $f(a)-b\ne f(a)-a\ne f(b)-a$.
It follows that if $f,g\in\F_q[x]$ are distinct orthomorphism
polynomials then $H(f,g)\ge3$. We investigate when this bound is tight. 
Our second main result is as follows:

\begin{thm}\label{t:Ham3}
There exist orthomorphism polynomials $f,g\in\F_q[x]$ that satisfy
$H(f,g)=3$ if and only if $q \notin \{2, 5, 8\}$.
\end{thm}

Cavenagh and Wanless \cite{cavenagh2010number} showed the special case
of \tref{t:Ham3} in which $q$ is prime. Their motivation was an application
to Latin bitrades that we discuss in the next section.

Suppose that $q-1=nk$, for some positive integers $n, k$. Let $\gamma$
be a primitive element of $\F_q^*$. Then we define
$C_{j, n} =\{\gamma^{ni+j} : 0 \leq i \leq k - 1\}$ to be a
\emph{cyclotomic coset} of the unique subgroup $C_{0,n}$ of index $n$
in $\F_q^*$.
A cyclotomic map $\psi_{a_0, \dots, a_{n - 1}}$ of index $n$ can then be defined
by
\begin{equation}\label{e:cyceqn}
\psi_{a_0, \dots, a_{n - 1}}(x) = \begin{cases}
0 &\text{if } x = 0, \\
a_ix &\text{if } x \in C_{i,n},
\end{cases}
\end{equation}
where $a_0, \dots, a_{n - 1} \in \F_q$. 
An orthomorphism is
\emph{non-cyclotomic} if it cannot be written as a cyclotomic
map for any index $n<q-1$.
We define a \emph{translation} $T_g$ of an
orthomorphism $\theta$ to be the orthomorphism $T_g[\theta](x) =
\theta(x + g) - \theta(g)$.
We say that an orthomorphism $\theta$ is \emph{irregular} if
$T_g[\theta]$ is non-cyclotomic for all $g\in\F_q$.
It was conjectured in \cite{FW17} that
irregular orthomorphisms exist over all sufficiently large fields.
We prove this and more in our last main result:

\begin{thm}\label{t:irreg}
There are irregular orthomorphisms over $\F_q$ for
$7<q\not\equiv1\bmod3$ and for even $q>4$.
For fields of odd characteristic, asymptotically almost all orthomorphisms
are irregular.
\end{thm}

Note that the $q=2^{2k+1}$ subcase of \tref{t:irreg} was already shown
in \cite{FW17}.

The structure of this paper is as follows. In the next section we
provide several different constructions for orthomorphisms that are as
close as possible to each other in Hamming distance.  The proofs of
our main results are given in \sref{s:summary}. Then in \sref{s:conclude}
we offer two conjectures for future research.

\section{Orthomorphisms at minimal Hamming distance}

In this section we provide several different methods for producing
pairs of orthomorphisms that are as close to each other as possible,
in Hamming distance. None of our methods work for all fields, but
together our methods will combine in \sref{s:summary} to prove
\tref{t:Ham3}.  \tref{t:bndach} will then follow immediately given the
next observation.

\begin{lem}\label{l:polydiff}
Suppose that $f, g \in \F_q[x]$ are reduced orthomorphism polynomials,
where $q>3$.  If $H(f,g)=3$, then at least one of $f$ or $g$ must have
degree $q-3$.
\end{lem}

\begin{proof}
Let $h = f - g$. Then $\deg(h) \leq \max\{\deg(f), \deg(g)\}\leq q-3$
by \tref{t:reddegorth}. Now $h$ is nonzero but has $q - 3$ roots, so
$\deg(h) \geq q - 3$. It follows that $\max\{\deg(f), \deg(g)\} = q-3$.
\end{proof}

Suppose that orthomorphism polynomials $f,g\in\F_q[x]$
satisfy $H(f,g)=k$. Define
\begin{align*}
L_1&=\big\{(i,f(j)-j+i,f(j)+i):i,j\in\F_q,\ f(j)\ne g(j)\big\},\\
L_2&=\big\{(i,g(j)-j+i,g(j)+i):i,j\in\F_q,\ f(j)\ne g(j)\big\}.
\end{align*}
Then it is easy to check that $L_1$ and $L_2$ are disjoint sets of
$kq$ ordered triples each, such that
\begin{itemize}
\item The projection of $L_1$ onto any two coordinates is $1$-to-$1$,
and its image equals the image of the same projection acting on $L_2$.
\item The projection of $L_1$ (or $L_2$) onto any one coordinate is
$k$-to-$1$.
\end{itemize}
These are exactly the conditions that mean that the pair $(L_1,L_2)$
forms what is called a \emph{$k$-homogeneous Latin bitrade}
(cf.~\cite{cavenagh2010number}). Hence all the methods in the following
subsections can be applied to construct $3$-homogeneous Latin bitrades.
For most fields, we will end up giving several different construction
methods.

\subsection{Fields of characteristic $p \notin \{2, 5\}$.}

Our first method works when the characteristic of the field is not
equal to $2$ or $5$.

\begin{thm}\label{non25}
Suppose that $\F_q$ has characteristic $p\notin\{2,5\}$.
Then there exist orthomorphism polynomials $f,g\in\F_{q}[x]$
with $H(f,g)=3$.
\end{thm}

\begin{proof}
For $q=3$ the orthomorphism polynomials $f=2x$ and $g=2x + 1$ suffice.
Thus we may assume that $q=p^r>5$. Let $P$ be the prime subfield of $\F_q$.
Cavenagh and Wanless \cite{cavenagh2010number} showed that there exist
orthomorphisms $\phi, \theta\in P[x]$ such that $H(\phi,\theta)=3$.
If $r=1$ then we are done, so assume that $r>1$.
Define the following maps on $\F_q$,
\begin{equation*}
f(x)=
\begin{cases} 
2x &\text{if } x\notin P, \\ 
\phi(x) &\text{if } x\in P,
\end{cases}
\end{equation*}
\begin{equation*}
g(x)=
\begin{cases} 
2x &\text{if } x\notin P, \\ 
\theta(x) &\text{if } x\in P.
\end{cases}
\end{equation*}
We know that the map $x\mapsto 2x$ permutes $\F_q$, and it clearly maps $P$ to $P$, so it follows that it also permutes $\F_q\backslash P$. By assumption $\phi$ permutes $P$. Hence $f$ is a permutation. By similar arguments $f$
is an orthomorphism, and so is $g$. Also
$H(f,g)=H(\phi,\theta)=3$.
\end{proof}

\subsection{Fields of order $1\bmod 3$}

Our next method works for fields of order $q\equiv 1 \bmod 3$.

\begin{thm}\label{1mod3}
Let $q \equiv 1 \bmod 3$.  Then there exist orthomorphism polynomials
$f,g\in\F_{q}[x]$ with $H(f,g)=3$.
\end{thm}

\begin{proof}
Let $q - 1 = 3k$, for some positive $k \in \mathbb{Z}$. 
Niederreiter and Winterhof \cite{niederreiter2005cyclotomic} showed that there
exists a ``near-linear'' orthomorphism $f$ over $F_q$, where
\begin{equation*}
f(x) = \begin{cases}
a_0x & \text{if }x \in C_{0,k}, \\
a_1x & \text{if }x \notin C_{0,k},
\end{cases}
\end{equation*}
for distinct $a_0,a_1 \in \F_q\setminus\{0,1\}$.
Now let $g$ be the map defined by $g(x)=a_1x$, and note that $g$ is an
orthomorphism. Also $H(f,g)=|C_{0,k}|=3$.
\end{proof}

In many instances when we apply \lref{l:polydiff} we will not know
which of the two polynomials has degree $q-3$ (plausibly they both
have that degree). However, in \tref{1mod3} it is clear that
$\deg(g)=1$, so $\deg(f)=q-3$.

\subsection{Fields of large odd order}

Our next method works in all sufficiently large fields of odd
characteristic. We will use the following auxiliary result.

\begin{lem}\label{l:partialorth}
Suppose that $\theta$ is an orthomorphism satisfying $\theta(0) = 0$, $\theta(b) = c$ and $\theta(c) = c - b$, for some distinct $b,c \in \F_q^*$. Then there exists an orthomorphism polynomial $\phi\in\F_q[x]$ such that $H(\theta,\phi)=3$.
\end{lem}

\begin{proof}
Define $\phi : \F_q \rightarrow \F_q$ by,
\begin{equation*}
\phi(x) = \begin{cases}
\theta(x) & \text{if }x \in \F_q \backslash \{0, b, c\}, \\ 
c - b & \text{if }x = 0, \\
c & \text{if }x = c, \\
0 & \text{if }x = b.
\end{cases}
\end{equation*}
Clearly $H(\theta,\phi)=3$ and
$\{\phi(0), \phi(b), \phi(c)\}
= \{c - b, 0, c\} = \{\theta(0), \theta(b), \theta(c)\}$.
So  $\phi$ is injective by the injectivity of $\theta$. 
Similarly,
\begin{equation*}
\phi(x) - x = \begin{cases}
\theta(x) - x & \text{if }x \in \F_q \backslash \{0, b, c\}, \\
c - b & \text{if }x = 0, \\
0 & \text{if }x = c, \\
-b & \text{if }x = b,
\end{cases}
\end{equation*}
and $\{\phi(0) - 0, \phi(b) - b, \phi(c) - c\}=\{c - b, -b, 0\}
= \{\theta(b) - b, \theta(c) - c, \theta(0) - 0\}$, so
$x \mapsto \phi(x) - x$ is injective because $\theta$ is an
orthomorphism. The result follows.
\end{proof}

To make use of \lref{l:partialorth} we need to find orthomorphisms
that have three specific values. Luckily, the next result does the
work for us. It is due to Cavenagh, H{\"a}m{\"a}l{\"a}inen and Nelson
\cite{cavenagh2009completing}, who stated it only for prime fields of
odd order. Their proof generalises without change to all fields of odd
order, so it will not be repeated here.

\begin{thm}\label{t:parorth}
Let $q \geq 191$ be an odd prime power. Let $z, k \in \F_q \backslash \{0, 1\}$ and $e \in \F_q \backslash \{0, z, k, k + z - 1\}$. Then there exists an orthomorphism $\theta$ over $\F_q$ satisfying $\theta(0) = 0$, $\theta(1) = z$ and $\theta(k) = e$.
\end{thm}

Applying \tref{t:parorth} in combination with \lref{l:partialorth}, we
obtain the result for this subsection:

\begin{thm}\label{largeq}
Let $q \geq 191$ be an odd prime power.
Then there exist orthomorphism polynomials
$f,g\in\F_{q}[x]$ with $H(f,g)=3$.
\end{thm}

\subsection{Fields of order $q = 2^r$ for odd $r$}

Our final method deals with the case of fields of order $2^r$ for odd integers $r$. We will need the following result of Williams \cite{williams1975note} regarding the reducibility of a cubic polynomial.

\begin{lem}\label{l:cubred}
Let $\F_q$ be a finite field of even order $q > 2$. Then the polynomial $f \in \F_q[x]$ defined by $f(x) = x^3 + ax + b$ has a unique root in $\F_q$ if and only if $\Tr(a^3b^{-2}) \neq \Tr(1)$.
\end{lem}

We will also need the following construction of an orthomorphism of a finite field of even order.
Let $q = 2^r$ for some integer $r \geq 3$ and let $a \in \F_q \backslash \{0, 1\}$. Define $H = \{0, 1, a, a + 1\}$ and let $c \in \F_q \backslash H$.
In \cite{FW17} it is shown that the map,
\begin{equation}\label{e:orthchar2}
\theta_a(x) = \begin{cases}
ax + a(a + 1) & \text{if }x \in H + c, \\
ax & \text{otherwise},
\end{cases}
\end{equation}
is an orthomorphism over $\F_q$ satisfying $\theta_a(0)=0$.

We now prove that when $q \geq 32$ is an odd power of $2$, there exist two orthomorphism polynomials over $\F_q$ at Hamming distance 3 from each other.

\begin{thm}\label{t:oddtwo}
Let $q = 2^r$ for some odd integer $r \geq 5$.  Then there exist
orthomorphism polynomials $f,g\in\F_{q}[x]$ with $H(f,g)=3$.
\end{thm}

\begin{proof}
There are $2^{r - 1} - 1$ non-zero elements with zero trace in $\F_q$, and the map $x \mapsto x^{-3}$ is a permutation of $\F^*_q$. It follows that there are $2^{r - 1} - 1$ choices of an element $c \neq 0$ such that $\Tr(c^{-3})=0$. As $2^{r - 1}-1>3$, there exists some $c \in \F_q^*$ such that $c^3 + c + 1 \neq 0$ and $\Tr(c^{-3}) = 0$. Define the polynomial $g \in \F_q[x]$ by $g(x) = x^3 + (c^2 + c + 1)x + c^2$. Note that
\begin{equation*}
\begin{aligned}
\Tr\big((c^2 + c + 1)^3c^{-4}\big)
&= \Tr(c^2) + \Tr(c) + \Tr(c^{-1}) + \Tr(c^{-4}) + \Tr(c^{-3}) 
= 0 \neq \Tr(1),
\end{aligned}
\end{equation*}
using the fact that $\Tr(a) = \Tr(a^2)= \Tr(a^4)$ for all $a \in \F_q$. Hence, \lref{l:cubred} implies that $g$ has a root. Let $f \in \F_q[x]$ be defined by $f(x) = g(x + c + 1) = x^3 + (c + 1)x^2 + cx + c$. It follows that $f$ also has a root, say $a$. Define $H = \{0, 1, a, a + 1\}$ and note that $a \notin \{0, 1, c, c + 1\}$ as none of these are roots of $f$. Since $a \notin \{0, 1\}$ and $c \notin H$, we can define an orthomorphism $\theta_a$ by \eref{e:orthchar2}.
Define $b = \frac ca$. We claim that $b \notin H + c$. If $b = c$ then $a = 1$, a contradiction. If $b = c + 1$ then $a = \frac c{c + 1}$ and so $0 = f(a) = f(\frac c{c + 1}) = c(c^3 + c + 1)(c + 1)^{-3}$, hence $c^3 + c + 1 = 0$, a contradiction. If $b = a + c$ then it follows that $c = \frac{a^2}{a + 1}$ and so $f(a) = \frac{a^3}{a + 1} = 0$, thus $a = 0$, a contradiction. If $b = a + c + 1$ then $c = a$ or $a = 1$, a contradiction. Thus $\theta_a(b) = ab = c$.
Now as $a$ is a root of $f$, we know that $a(ac + a(a + 1) + c + \frac ca) = 0$, and hence $\theta_a(c) = ac + a(a + 1) = c + \frac ca = c + b$. 
The result now follows from \lref{l:partialorth}.
\end{proof}

\section{Proof of the main results}\label{s:summary}

We are now in a position to prove our main results.

\begin{proof}[Proof of \tref{t:Ham3}]
There clearly cannot exist polynomials which have Hamming distance $3$
over $\F_2$. By \lref{l:polydiff}, if there existed polynomials $f, g
\in \F_q[x]$ with $H(f, g) = 3$ when $q \in \{5, 8\}$, then there
would exist orthomorphism polynomials of degree $q-3$ over these
fields, which we know is not the case from
\cite{shallue2013permutation}. It remains to justify the claim that
$f$ and $g$ exist for all prime powers $q = p^r \notin \{2, 5, 8\}$.

If $p \notin \{2, 5\}$ then the claim is true by \tref{non25}. If $p \in \{2, 5\}$ and $r$ is even then the claim follows from \tref{1mod3}. If $r$ is odd and $p=2$ then the claim follows from \tref{t:oddtwo}. If $q=5^r \geq 191$, then the claim follows from \tref{largeq}. The only remaining case is $q = 125$. Consider $\F_{125}$ as $\mathbb{Z}_5[y] / (y^3 + 3y + 3)$, and note that $y$ is a primitive element of $\F_{125}^*$. Let $a = y^2 \notin C_{0, 4}$ and $b = y^2 + 4 = y^{75} \notin C_{0, 4}$. Then $f \in \F_{125}[x]$ defined by $f(x) = (a - b)^{-1}x^5 - b(a - b)^{-1}x$ is an orthomorphism polynomial by a result of Niederreiter and Robinson \cite{niederreiter1982complete}. Furthermore $f$ satisfies $f(0) = 0$, $f(y^2) = y^{118} = 4y^2 + 3$ and $f(y^{118}) = y^{40} = 3y^2 + 3 = y^{118} - y^2$, and so we are done, by \lref{l:partialorth}.
\end{proof}

The only fields not having orthomorphisms at Hamming distance 3 are
$\F_2$, $\F_5$ and $\F_8$. There are no orthomorphisms at all over
$\F_2$. Over $\F_5$, it was noted in \cite{cavenagh2010number} that
the minimum Hamming distance between orthomorphisms is 4 (this
distance is achieved by $f=2x$ and $g=3x$). The minimum Hamming
distance between orthomorphisms over $\F_8$ is also $4$. To see this,
consider $f=ax$ and $g=\theta_a$ from \eref{e:orthchar2}, and apply
the logic behind \lref{l:polydiff}.

\begin{proof}[Proof of \tref{t:bndach}]
Orthomorphisms of degree $q-3$ do not exist when
$q \in \{2, 3, 5, 8\}$ as shown in \cite{shallue2013permutation}.
For all other prime powers $q$, existence of
orthomorphisms of degree $q-3$ follows by combining \tref{t:Ham3} with
\lref{l:polydiff}.
\end{proof}

Note that from \cite{shallue2013permutation}, the maximum reduced
degree of any orthomorphism polynomial over $\F_3$, $\F_5$ and $\F_8$
is respectively $1$, $1$ and $4$.

\begin{proof}[Proof of \tref{t:irreg}]
Eberhard, Manners and Mrazovi\'{c} \cite{EMM19} showed that any
abelian group of odd order $q$ has $(e^{-1/2}+o(1))q!^2q^{1-q}$
orthomorphisms. In particular, this is true for the additive group of
$\F_q$ (when $q$ is odd). However, the number of cyclotomic maps
of index $i$ over $\F_q$ is at most $q^i$, given that the map is
determined by the values of $a_1,\dots,a_i$ in
\eref{e:cyceqn}. Each such map has $q$
translations and the only relevant values of $i$ satisfy $1\le i<q/2$.
Hence the number of orthomorphisms over $\F_q$ that are not irregular is
at most $q^{q/2}q(q/2)=q^{q/2+2}/2$, which is asymptotically
insignificant compared to the total number of orthomorphisms.  We
conclude that for fields of odd characteristic, asymptotically almost
all orthomorphisms are irregular.

Next, suppose that $q=2^r>4$ and consider the orthomorphism
$\theta_a$ defined by \eref{e:orthchar2}. In \cite[Thm~5]{FW17} it was
shown that $\theta_a$ is irregular if $r$ is odd. However, examining
that proof reveals that $\theta_a$ will also be irregular for even $r$,
provided that the set $X_g=\{g+1,g+a,g+a+1\}$ is not a union of
cyclotomic cosets for any $g\in\F_q$. As $|X_g|=3$, the only
possible problem is that $X_g=C_{i,n}$ for some $i$, where $n=(q-1)/3$. Note that
$X_g$ contains two elements that differ by $1$ and the third element
differs from one of those two elements by $a$. Hence for each $i$ there
are at most 2 choices of $a$ that might allow $X_g=C_{i,n}$ to be satisfied for
some $g$.
Eliminating these choices for all $i$, we lose at most $2(q-1)/3<q-2$ choices
for $a$. Hence, we can pick a value for $a$ such that
$X_g\ne C_{i,n}$ for all $g\in\F_q$ and $0\le i<n$. In that case,
$\theta_a$ is irregular.

Finally, we consider the case when $7<q\not\equiv1\bmod3$. The
previous case handled $q=8$, so assume $q>8$.  Now, \tref{t:bndach}
ensures the existence of an orthomorphism $\theta$ of reduced degree
$q-3$.  Suppose that $T_g[\theta]$ is cyclotomic of index $i<q-1$ for
some $g\in\F_q$. Note that $T_g[\theta]$ also has reduced degree
$q-3$.  Hence, by \cite[Thm~1]{niederreiter2005cyclotomic} we must
have $\gcd(q-1,3)>1$, but that contradicts the fact that
$q\not\equiv1\bmod3$.
\end{proof}

\section{Concluding remarks}\label{s:conclude}

For each finite field we have established what the minimum distance between
two distinct orthomorphisms is, and what the largest degree of a reduced
orthomorphism polynomial is. 

A direction for further research would be to investigate the proportion of
orthomorphisms that have reduced degree $q-3$. It was shown in
\cite{MR1933625} that asymptotically almost all permutation
polynomials in $\F_q[x]$ have reduced degree $q-2$.  The data for
small fields presented in \cite{shallue2013permutation} is consistent
with orthomorphisms displaying a similar trend. We propose:

\begin{conj}\label{cj:mostq-3}
Asymptotically almost all orthomorphism polynomials in $\F_q[x]$ 
have reduced degree $q-3$, as $q\to\infty$.
\end{conj}

On the subject of typical behaviour for orthomorphisms of large fields,
we showed that almost all orthomorphisms of fields of odd order are irregular.
We believe that fields of even order share this property.

\begin{conj}\label{cj:irreg}
Asymptotically almost all orthomorphisms are irregular.
\end{conj}

A related open question is to establish what is the largest field without
irregular orthomorphisms. We know from \tref{t:irreg} that the answer is
a field of odd order (it may well be $\F_7$).

\bibliographystyle{plain}
\bibliography{orth_references}

\end{document}